\input ./style/arxiv-vmsta.cfg
\documentclass[numbers,compress,v1.0.1]{vmsta}

\usepackage{lscape}

\volume{2}
\issue{4}
\pubyear{2015}
\firstpage{421}
\lastpage{441}
\doi{10.15559/15-VMSTA45}% Updated by VTEXPTS2LaTeX.exe, 22.12.2015
\startlocaldefs

\allowdisplaybreaks
\endlocaldefs

\newtheorem{thm}{Theorem}

\theoremstyle{definition}
\newtheorem{defin}{Definition}

\begin{document}
\begin{frontmatter}

\title{Ruin probability in the three-seasonal discrete-time risk model}

\author{\inits{A.}\fnm{Andrius}\snm{Grigutis}\fnref{f1}}\email{andrius.grigutis@mif.vu.lt}
\fntext[f1]{The first author was supported by grant No. MIP-049/2014
from the Research Council of Lithuania.}
\author{\inits{A.}\fnm{Agne\v{s}ka}\snm{Korvel}}\email{agneska.korvel@mif.vu.stud.lt}
\author{\inits{J.}\fnm{Jonas}\snm{\v Siaulys}\corref{cor1}}\email{jonas.siaulys@mif.vu.lt}
\cortext[cor1]{Corresponding author.}
\address{Faculty of Mathematics and Informatics, Vilnius University,
Naugarduko 24, Vilnius~LT-03225, Lithuania}

\markboth{A. Grigutis et al.}{Ruin probability in the three-seasonal
discrete-time risk model}

\begin{abstract}
This paper deals with the discrete-time risk model with nonidentically
distributed claims. We suppose that the claims repeat with time periods
of three units, that is, claim distributions coincide at times
$\{1,4,7,\ldots\}$, at times $\{2,5,8,\ldots\}$, and at times
$\{3,6,9,\ldots\}$. We present the recursive formulas to calculate the
finite-time and ultimate ruin probabilities. We illustrate the
theoretical results by several numerical examples.
\end{abstract}

\begin{keyword}
Inhomogeneous model\sep
three-seasonal model\sep
finite-time ruin probability\sep
ultimate ruin probability
\MSC[2010] 91B30\sep60G50
\end{keyword}

\received{26 November 2015}% Updated by VTEXPTS2LaTeX.exe, 22.12.2015
%13:51
\revised{14 December 2015}% Updated by VTEXPTS2LaTeX.exe, 22.12.2015
%13:51
\accepted{14 December 2015}% Updated by VTEXPTS2LaTeX.exe, 22.12.2015
%13:51
\publishedonline{30 December 2015}
\end{frontmatter}

\section{Introduction}\label{i}

The discrete-time risk model is a classical collective risk model for
insurance. In the homogeneous version of this model, the
insurer's surplus at each time $n\in\mathbb{N}_0=\{0,1,2,\ldots\}$ is
defined by the following equality:
\begin{equation}
\label{a1} W_u(n)=u+n-\sum\limits
_{i=1}^{n}Z_i,
\end{equation}
where $u\in\mathbb{N}_0$ is the initial insurer's surplus, and the
claim amounts $Z_1, Z_2, \ldots$ are assumed to be independent copies
of a nonnegative
integer-valued random variable~$Z$. This random variable and the
initial surplus $u$ generate the homogeneous discrete-time
risk model. A typical path of the surplus process $W_u(n)$ is shown in Fig.~\ref{fig1}.

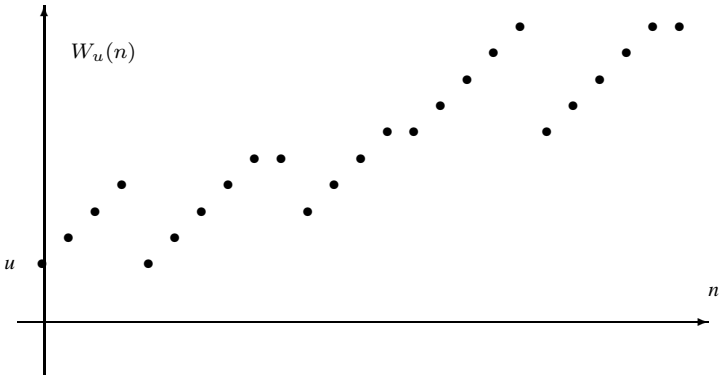
\begin{figure}[h]
\begin{center}
\begin{picture}(250,130) \put(250,10){\textit{n}}
\put(-15,20){\textit{u}} \put(10,100){${W_u}(n)$}
\put(-10,0){\vector(1,0){260}}
\put(0,-20){\vector(0,1){140}} \put(-3,20){$\bullet$}
\put(7,30){$\bullet$} \put(17,40){$\bullet$}
\put(27,50){$\bullet$}
\put(37,20){$\bullet$} \put(47,30){$\bullet$}
\put(57,40){$\bullet$} \put(67,50){$\bullet$}
\put(77,60){$\bullet$} \put(87,60){$\bullet$}
\put(97,40){$\bullet$}
\put(107,50){$\bullet$} \put(117,60){$\bullet$}
\put(127,70){$\bullet$}
\put(137,70){$\bullet$}\put(147,80){$\bullet$}\put(157,90){$\bullet
$}\put(167,100){$\bullet$}\put(177,110){$\bullet$}
\put(187,70){$\bullet$}
\put(197,80){$\bullet$}\put(207,90){$\bullet$}\put(217,100){$\bullet
$}\put(227,110){$\bullet$}\put(237,110){$\bullet$}

%\put(247,110){$\bullet$}\put(257,100){$\bullet$}\put(267,80){$\bullet$}
%\put(277,90){$\bullet$}\put(287,100){$\bullet$}
\end{picture}
\vspace*{12pt}
\end{center}
\caption{Behavior of the surplus sequence $W_u(n)$}
\label{fig1}
\end{figure}

The claim amount generator $ Z$ can be characterized by the
probability mass function (p.m.f.)
\[
z_k=\mathbb{P}(Z=k), \quad k\in\mathbb{N}_0,
\]
or by the cumulative distribution function (c.d.f.)
\[
F_Z(x)=\sum\limits_{k=0}^{\lfloor x\rfloor}z_k, \quad x\in\mathbb{R},
\]
where $\lfloor x\rfloor$ denotes the integer part of $x$.

The homogeneous discrete-time risk model has been extensively
investigated by De Vylder and Goovaerts \cite{dv+goo-1988,dv+goo-1999}, Dickson \cite{di-1999,di-2005}, Gerber \cite
{ger}, Seal \cite{seal},
Shiu \cite{shiu-a,shiu-i}, Picard and Lef\`{e}vre \cite
{pl-1997,pl-2003}, Lef\`{e}vre and Loisel \cite{ll}, Leipus and
\v{S}iaulys \cite{ls}, Tang \cite{tang}, and other authors.
The ruin time, the ultimate ruin probability, and the finite-time ruin
probability are the main extremal characteristics of any risk model.
The first time $T_u$ when the surplus $W_u(n)$ becomes negative or null
is called the ruin time, that is,
%\vspace{-2mm}
%
\begin{equation*}
T_u= %
\begin{cases}
\inf\{n\in\mathbb{N}:W_u(n)\leqslant0\},  \\
\infty  \ \mbox{ if}\  W_u(n)>0 \ \mbox{for all } \ n\in\mathbb{N}.
\end{cases} %
\end{equation*}
The ruin probability until time $T\in\mathbb{N}$ is called the \textit
{finite-time ruin probability} and is defined by
\[
\psi(u,T)=\mathbb{P}(T_u\leqslant T).
\]
The infinite-time or \textit{ultimate ruin probability} is defined by
\[
\psi(u)=\mathbb{P}(T_u<\infty).
\]
So, for the \textit{ultimate survival probability}, we have
\[
\varphi(u)=1-\psi(u)=\mathbb{P}(T_u=\infty).
\]
The presented definitions imply that
\begin{align*}
&\psi(u,T)=\mathbb{P} \Biggl(\bigcup\limits_{n=1}^{T} \Biggl\{u+n-\sum\limits_{i=1}^{n}Z_i\leqslant0 \Biggr\} \Biggr)=\mathbb{P} \Biggl(\max\limits_{1\leqslant n\leqslant T}\sum\limits_{i=1}^{n}(Z_i-1)\geqslant u \Biggr),\\
&\psi(u)=\mathbb{P} \Biggl(\bigcup\limits_{n=1}^{\infty} \Biggl\{u+n-\sum\limits_{i=1}^{n}Z_i\leqslant0 \Biggr\}\Biggr)=\mathbb{P} \Biggl(\sup\limits_{n\geqslant1}\sum\limits_{i=1}^{n}(Z_i-1)\geqslant u \Biggr),\\
&\varphi(u)=\mathbb{P} \Biggl(\bigcap\limits_{n=1}^{\infty} \Biggl\{u+n-\sum\limits_{i=1}^{n}Z_i> 0 \Biggr\}\Biggr),\\
&\lim\limits_{T\nearrow\infty}\psi(u,T)=\psi(u).
\end{align*}

Several formulas and procedures for computing finite-time ruin
probability and ultimate ruin probability have been proposed in the
literature. Here we
present some of them having the recursive form.

\begin{itemize}
\item \textit{For the homogeneous discrete-time risk model, we have
\textup{(}see, for instance, \textup{\cite{dv+goo-1988,di-2005,dw-1991}):}
\begin{align*}
\psi(u,1)&=1-F_Z(u), \quad u\in\mathbb{N}_0,\\
\psi(u,T)&=\psi(u,1)+\sum\limits_{k=0}^{u}\psi(u+1-k,T-1)z_k, \quad\! u\in \mathbb{N}_0, \quad\! T\in\{2,3,\ldots\}.
\end{align*}}
\item \textit{If model \eqref{a1} is generated by the claim
generator $Z$ such that $\mathbb{E}Z<1$, then the ultimate ruin
probability can be calculated by the formulas \textup{(}see, for instance,
\textup{\cite{di-2005,dw-1991, shiu-i}):}
\begin{align}
\psi(0)&=\mathbb{E}Z,\label{a2}\\
\psi(u)&=\sum\limits_{j=1}^{u-1}\bigl(1-F_Z(j)\bigr)\psi(u-j) +\sum\limits_{j=u}^{\infty}\bigl(1-F_Z(j)\bigr), \quad u\in\mathbb{N}.\label{a3}
\end{align}}
\end{itemize}

If the homogeneous discrete-time risk model is generated by $Z$
satisfying condition $\mathbb{E}Z\geqslant1$, then we say that  the net profit
condition does not hold, and, in such a case, we
have that $\psi(u)=1$ for all $u\in\mathbb{N}_0$ according to the
general renewal theory (see, e.g., \cite{mi-2009} and
\xch{the references}{references} therein).

The formulas presented enable us to calculate $\psi(u)$ and $\psi(u,T)$
for $u\in\mathbb{N}_0$ and $T\in\mathbb{N}$. Nevertheless, there exist
many other methods that allow us to calculate or estimate the
finite-time and the ultimate ruin probabilities. Some of them can be
found in \cite{aa-2010, llg-2009, pl-2003}.

The assumption for claim amounts $\{Z_1,Z_2,\ldots\}$ to be
nonidentically distributed random variables is a natural generalization
of the homogeneous model. If r.v.s\break $\{Z_1,Z_2,\ldots\}$ are
independent but not necessarily identically distributed, then the\break
model defined by Eq.~\eqref{a1} is called the \textit{inhomogeneous
discrete-time risk model}. For such a model, a recursive procedure
for calculation of finite-time ruin probabilities can be found in \cite{bs,bbs}.
For the finite-time ruin probabilities
\[
\psi^{(j)}(u,T)=\mathbb{P} \Biggl(\bigcup\limits
_{n=1}^{T}
\Biggl\{ u+n-\sum\limits
_{i=1}^{n}Z_{i+j}\leqslant0
\Biggr\} \Biggr), \quad j\in\mathbb{N}_0,
\]
we have the following theorem.
\begin{thm}\label{t1}
Let us consider the inhomogeneous discrete-time risk model
defined by Eq.~\eqref{a1} in which $u\in\mathbb{N}_0$,
$z_k^{(j)}=\mathbb{P}(Z_{1+j}=k)$, $k,j\in\mathbb{N}_0$, and $
F_Z^{(j)}(x)=\mathbb{P}(Z_{1+j}\leqslant x)$, $x\in\mathbb{R}$. Then
\begin{align*}
\psi^{(j)}(u,1)&=1-F_Z^{(j)}(u),\\
\psi^{(j)}(u,T)&=\psi^{(j)}(u,1)+\sum\limits_{k=0}^{u}\psi ^{(j+1)}(u+1-k,T-1)z_k^{(j)}
\end{align*}
for all $u\in\mathbb{N}_0$ and $ T\in\{2,3,\ldots\}$.
\end{thm}

According to this theorem, we can calculate the finite-time ruin
probability\break $ \psi^{(0)}(u,T)$ of the initial model for all
$u\in\mathbb{N}_0$ and $T\in\mathbb{N}$. Unfortunately, it is
impossible to get formulas for $\psi(u)$ similar to formulas \eqref{a2} and \eqref
{a3} in the general case because in the case of nonidentically
distributed claims, the future of model behavior at each time can be
completely new. In paper \cite{jj-2015}, the general discrete-time risk
model was restricted to the model with two kinds of claims. In this model,
there are two differently distributed claim amounts that are changing
periodically. We call such a model the bi-seasonal discrete-time risk
model. In \cite{jj-2015} (see Theorem 2.3), the following statement is
proved for the calculation of the ultimate ruin probability.

\begin{thm}\label{t2}
Let us consider a bi-seasonal discrete-time risk model
generated by independent random claim amounts $X$ and $Y$, that is,
$Z_{2k+1}\mathop{=}\limits^{d}X$ for $k\in\{0,1, \ldots\}$ and
$Z_{2k}\mathop{=}\limits^{d}Y$ for $k\in\{1,2,\ldots\}$. Denote
$S=X+Y$ and $x_n=\mathbb{P}(X=n)$,
$y_n=\mathbb{P}(Y=n)$, $s_n=\mathbb{P}(S=n)$ for $n\in\mathbb{N}_0=\{
0,1,\ldots\}$.
\begin{itemize}
\item If $\mathbb{E}X+\mathbb{E}Y<2$, then
\[\lim_{u\rightarrow\infty}\psi(u)=0.\]
\item If $s_0=x_0y_0\neq0$, then:
\begin{align*}
\psi(0)&=1-(2-\mathbb{E}S)\lim\limits
_{n\rightarrow\infty}\frac
{b_{n+1}-b_n}{a_n-a_{n+1}},
\\
1-\psi(u)&= \alpha_u\bigl(1-\psi(0)\bigr)+\beta_u(2-
\mathbb{E}S), \quad u\in\mathbb{N},
\end{align*}
where $\{\alpha_n\}$, $\{\beta_n\}$, $n\in\mathbb{N}_0$, are two
sequences of real numbers defined recursively by formulas:\vspace{-3pt}
\begin{align*}
& \alpha_0=1, \qquad \alpha_1=-\frac{1}{y_0}, \\
&\alpha_n=\frac{1}{s_0} \Biggl(\alpha_{n-2}-\sum\limits_{i=1}^{n-1}s_i\alpha_{n-i}-x_{n-1}\Biggr), \quad n\geqslant2,\\
&\beta_0=0, \qquad \beta_1=\frac{1}{y_0}, \\
&\beta_n=\frac{1}{s_0} \Biggl(\beta_{n-2}-\sum\limits_{i=1}^{n-1}s_i\beta_{n-i}+x_{n-1} \Biggr),\quad n\geqslant2.
\end{align*}
\item If $s_0\neq0$, then
\begin{align*}
\psi(1)&= 1-\bigl(1+\psi(0)-\mathbb{E}S\bigr)/y_0,\\
\psi(u)&= 1+\frac{1}{s_0} \Biggl(\psi(u-2)-1+\sum\limits_{k=1}^{u-1}s_k\bigl(1-\psi(u-k)\bigr) \Biggr)\\
&\quad -\frac{x_{u-1}(1-\psi(1))}{x_0},\quad u\in\{2,3,\ldots\}.
\end{align*}

\item If $x_0=0$, $y_0\neq0$, then $s_1\neq0$ and $\psi (0)=1$.
\item If $x_0\neq0$, $y_0=0$, then $s_1\neq0$ and $\psi (0)=\mathbb{E}S-1$.
\item If $s_0=0$, then, for $u\in\mathbb{N}$,
\[
\psi(u)=1-\frac{1}{s_1} \Biggl(1-\psi(u-1)-\sum\limits
_{k=2}^{u}s_k
\bigl(1-\psi (u-k+1)\bigr) \Biggr).
\]
\end{itemize}
\end{thm}

In this paper, we consider the discrete-time risk model with three
seasons. We obtain a list of formulas similar to those in Theorem~\ref
{t2} to calculate the ultimate ruin probability in such a model. In
Section~\ref{m}, we present a precise definition of the three-seasonal
discrete-time risk model and our main statements, whereas in Sections
\ref{pa} and~\ref{pa1}, we give detailed proofs. Finally, Section~\ref
{papa} deals with some numerical examples.

\section{Main results}\label{m}

We now present the model under consideration.
\begin{defin}
We say that the insurer's surplus $ W_u(n)$ follows the
three-seasonal risk model if $W_u(n)$ is given by Eq.~\eqref{a1} for each
$n\in\mathbb{N}_0$ and the following assumptions hold:
\begin{itemize}
\item the initial insurer's surplus $u\in\mathbb{N}_0$,
\item the random claim amounts $Z_1,Z_2,\ldots$ are nonnegative integer-valued independent r.v.s,
\item for all $k\in\mathbb{N}_0$, we have $Z_{3k+1}\mathop
{=}\limits^{d}Z_1$, $Z_{3k+2}\mathop{=}\limits^{d}Z_2$, and
$Z_{3k+3}\mathop{=}\limits^{d}Z_3$.
\end{itemize}
\end{defin}

Let us define p.m.f.s and p.d.f.s by the following equalities:
\begin{align*}
&a_k=\mathbb{P}(Z_1=k), \qquad b_k=\mathbb{P}(Z_2=k), \\
&c_k=\mathbb{P}(Z_3=k), \qquad s_k=\mathbb{P}(S=k), \quad k\in\mathbb{N}_0,
\end{align*}
where  $S=Z_1+Z_2+Z_3$,
\begin{align*}
&A(x)=\sum_{k=0}^{\lfloor x\rfloor}a_k, \qquad B(x)=\sum_{k=0}^{\lfloor x\rfloor}b_k, \\
&C(x)=\sum_{k=0}^{\lfloor x\rfloor}c_k, \qquad D(x)=\sum_{k=0}^{\lfloor x\rfloor}s_k, \quad x\geqslant0.
\end{align*}

It is not difficult to see that the definitions of ruin time,
finite-time ruin probability, ultimate ruin probability, and ultimate
survival probability remain the same.
All expressions of these quantities are the same as in the homogeneous
discrete-time risk model. However, the procedures to calculate the
finite-time or the ultimate probabilities are more complex than in the
homogeneous or be-seasonal discrete-time risk models.

Our first result immediately follows from Theorem~\ref{t1}. The
obtained formulas allow us to calculate the finite-time ruin
probabilities $\psi(u,T)=\psi^{(0)}(u,T)$ in the three-seasonal risk
model for all $u\in\mathbb{N}_0$ and all $T\in\mathbb{N}$.

\begin{thm}\label{t3}
In the three-seasonal discrete-time risk model, for each $u\in
\mathbb{N}_0$, we have
\[
\psi^{(0)}(u,1)=\sum\limits_{k>u}a_k, \qquad \psi^{(1)}(u,1)=\sum\limits_{k>u}b_k, \qquad \psi^{(2)}(u,1)=\sum\limits_{k>u}c_k,
\]
and for all $u\in\mathbb{N}_0$ and $ T\in\{ 2,3, \ldots\}$, we have the
following recursive formulas:\vspace{-3mm}
\begin{align*}
\psi^{(0)}(u,T)&=\psi^{(0)}(u,1)+\sum\limits_{k=0}^{u}\psi ^{(1)}(u+1-k,T-1)a_k,\\
\psi^{(1)}(u,T)&=\psi^{(1)}(u,1)+\sum\limits_{k=0}^{u}\psi ^{(2)}(u+1-k,T-1)b_k,\\
\psi^{(2)}(u,T)&=\psi^{(2)}(u,1)+\sum\limits_{k=0}^{u}\psi^{(0)}(u+1-k,T-1)c_k.
\end{align*}
\end{thm}

Our second result describes the meaning of the net profit condition in
the three-seasonal discrete-time risk model. The proof of the theorem
is presented in Section~\ref{pa}.

\begin{thm}\label{t4}
Consider the three-seasonal discrete time risk model generated
by independent random claim amounts $Z_1$, $Z_2$, and $Z_3$. If
\,$\mathbb{E}S> 3$, then $\psi(u)=1$ for each initial surplus $u\in
\mathbb{N}_0$. If \,$\mathbb{E}S=3$, then we have the following possible cases:
\begin{itemize}
\item $\psi(0)=\psi(1)=\psi(2)=1$ and $\psi(u)=0$ for $u\in\{3,4,\ldots\}$ if $\{a_3=b_0=c_0=1\}$;
\item $\psi(0)=\psi(1)=1$ and $\psi(u)=0$ for $u\in\{2,3,\ldots\}$ if $\{a_0=b_3=c_0=1\}$, $\{a_2=b_1=c_0=1\}$, $\{a_1=b_2=c_0=1\},$ or $\{a_2=b_0=c_1=1\}$;
\item $\psi(0)=1$ and $\psi(u)=0$ for all $u\in\mathbb{N}$ if $\{a_0=b_0=c_3=1\}$, $\{a_0=b_2=c_1=1\}$, $\{a_0=b_1=c_2=1\}$, $\{a_1=b_0=c_2=1\}$, or $\{a_1=b_1=c_1=1\}$;
\item $\psi(u)=1$ for all $u\in\mathbb{N}_0=\{0,1,2,\ldots\}$ if $s_3<1$.
\end{itemize}
\end{thm}

Our last statement proposes a recursive procedure for calculation of
the ultimate survival probabilities $\varphi(u)=1-\psi(u),\,u\in\mathbb{N}_0$.
The proof of the formulas is given in
Section~\ref{pa1}.

\begin{thm}\label{t5}
Consider the three-seasonal discrete-time risk model generated
by independent random claim amounts $Z_1$, $Z_2$, and $Z_3$. Denote
$S=Z_1+Z_2+Z_3$, $s_n=\mathbb{P}(S=n)$ for $n\in\mathbb{N}_0$, and
suppose that $\mathbb{E}S<3$. Then the following statements hold.
\begin{itemize}
\item $\lim\limits_{u\rightarrow\infty}\varphi(u)=1$.
\item If
$s_0\neq0$, then
\begin{equation}
\label{opop} \varphi(n)=\alpha_n\varphi(0)+\beta_n
\varphi(1)+\gamma_n(3-\mathbb {E}S), \quad n\in\mathbb{N}_0,
\end{equation}
where
\begin{equation*}
\begin{cases}\alpha_0=1, \qquad  \alpha_1=0, \qquad \alpha_2=-\frac{1}{b_0c_0},\\[2pt]
\alpha_n = \frac{1}{s_0} (\alpha_{n-3}-\sum_{k=1}^{n-1}s_k\alpha_{n-k}-a_{n-2} ), \quad n \geqslant3;
\end{cases} %
\end{equation*}
\begin{equation*}
\begin{cases}
\beta_0=0, \qquad \beta_1=1, \qquad \beta_2=-\frac{c_1}{c_0}-\frac{1}{b_0},\\[2pt]
\beta_n = \frac{1}{s_0} (\beta_{n-3}-\sum_{k=1}^{n-1}s_k\beta_{n-k}-a_{n-2}c_0
+c_0\sum_{k=0}^{n-1}a_kb_{n-1-k} ), \quad n \geqslant3;
\end{cases} %
\end{equation*}
\begin{equation*}
\begin{cases}
\gamma_0=0, \qquad \gamma_1=0, \qquad \gamma_2=\frac{1}{b_0c_0},\\[2pt]
\gamma_n = \frac{1}{s_0} (\gamma_{n-3}-\sum
_{k=1}^{n-1}s_k\gamma_{n-k}+a_{n-2} ), \quad n \geqslant3.
\end{cases} %
\end{equation*}
\item If $\{a_0=0,b_0\neq0, c_0\neq0, a_1\neq0\}$, then
\begin{equation}
\label{4+} %
\begin{cases}\varphi(0)=0,\\[2pt]
\varphi(n)=\hat{\beta}_n\varphi(1)+\hat{\gamma}_n(3-\mathbb{E}S), \quad n\in\mathbb{N},
\end{cases} %
\end{equation}
where
\begin{equation*}
\begin{cases}
\hat{\beta}_1=\beta_1, \qquad \hat{\beta}_2=\beta_2, \qquad \hat{\gamma}_1=\gamma_1, \qquad \hat{\gamma}_2=\gamma_2,\\[2pt]
\hat{\beta}_n=\frac{1}{s_1} (\hat{\beta}_{n-2}-\sum_{k=2}^{n}s_k\hat{\beta}_{n-k+1}\\[2pt]
\qquad\ -\,a_{n-1}c_0-c_0\varphi(1)\sum_{k=1}^{n}a_kb_{n-k} ), \quad n\geqslant3,\\[3pt]
\hat{\gamma}_n=\frac{1}{s_1} (\hat{\gamma}_{n-2}-\sum_{k=2}^{n}s_k\hat{\gamma}_{n-k+1}+a_{n-1} ), \quad n\geqslant3.
\end{cases} %
\end{equation*}
\item If $\{a_0\neq0,b_0=0, c_0\neq0, b_1\neq0\}$ then
\begin{equation}
\label{4++} \varphi(n)=\tilde{{\alpha}}_n\varphi(0)+\tilde{{
\gamma}}_n(3-\mathbb {E}S), \quad n\in\mathbb{N},
\end{equation}
where
\begin{equation*}
\begin{cases}
\tilde{{\alpha}}_1=-1/c_0, \qquad \tilde{\alpha}_2=c_1/c_0^2+1/(a_0b_1c_0), \\[2pt]
\tilde{\gamma}_1=1/c_0, \qquad \tilde{\gamma}_2=-c_1/c_0^2,\\[2pt]
\tilde{\alpha}_n=\frac{1}{s_1} (\tilde{\alpha}_{n-2}-\sum_{k=2}^{n}s_k\tilde{\alpha}_{n-k+1}-\sum_{k=0}^{n-1}a_kb_{n-k} ), \quad n\geqslant3,\\[2pt]
\tilde{\gamma}_n=\frac{1}{s_1} (\tilde{\gamma}_{n-2}-\sum_{k=2}^{n}s_k\tilde{\gamma}_{n-k+1}+\sum_{k=0}^{n-1}a_kb_{n-k} ), \quad n\geqslant3.
\end{cases} %
\end{equation*}
\item If $\{a_0\neq0,b_0\neq0, c_0=0, c_1\neq0\}$, then
\begin{equation}
\label{4+++} \varphi(n)=\breve{{{\alpha}}}_n\varphi(0)+\breve{{{
\gamma}}}_n(3-\mathbb {E}S), \quad n\in\mathbb{N}_0,
\end{equation}
where
\begin{equation*}
\begin{cases}
\breve{\alpha}_0=1, \qquad \breve{{\alpha}}_1=-1/(b_0c_1), \qquad \breve{\gamma}_0=0, \qquad \breve{\gamma}_1=1/(b_0c_1),\\[2pt]
\breve{\alpha}_n=\frac{1}{s_1} (\breve{\alpha}_{n-2}-\sum_{k=2}^{n}s_k\breve{\alpha}_{n-k+1}-\sum_{k=0}^{n-1}a_kb_{n-k} ), \quad n\geqslant2,\\[2pt]
\breve{\gamma}_n=\frac{1}{s_1} (\breve{\gamma}_{n-2}-\sum_{k=2}^{n}s_k\breve{\gamma}_{n-k+1}+\sum_{k=0}^{n-1}a_kb_{n-k} ), \quad n\geqslant2.
\end{cases} %
\end{equation*}
\item If $\{a_0=0,b_0=0, c_0\neq0\}$, then $\varphi(0)=0$,
$\varphi(1)=(3-\mathbb{E}S)/c_0$, and
\begin{align}
\label{4A} \varphi(u+1)&=\frac{1}{s_2} \Biggl((1-s_3)\varphi(u)-\sum\limits_{k=1}^{u-1}\varphi(k)s_{u+3-k}\nonumber\\
&\quad +c_0\varphi(1)\sum\limits_{k=0}^{u+2}a_kb_{u+2-k}\Biggr), \quad u\in\mathbb{N}.
\end{align}
\item If $\{a_0=0,b_0\neq0, c_0=0\}$, then $\varphi
(0)=s_2\varphi(1)$, $\varphi(1)=(3-\mathbb{E}S)/(s_2+b_0c_1)$, and
\begin{align}
\label{4B} \varphi(u+1)&=\frac{1}{s_2} \Biggl((1-s_3)\varphi(u)-\sum\limits_{k=1}^{u-1}\varphi(k)s_{u+3-k}\nonumber\\
&\quad +a_{u+1}b_0c_1\varphi(1)\Biggr), \quad u\in\mathbb{N}.
\end{align}
\item If $\{a_0\neq0,b_0=0, c_0=0\}$, then $\varphi
(0)=3-\mathbb{E}S $, $\varphi(1)=(3-\mathbb{E}S)/s_2$, and
\begin{eqnarray}
\label{4C} &&\varphi(u+1)=\frac{1}{s_2} \Biggl((1-s_3)
\varphi(u)+\sum\limits
_{k=1}^{u-1}\varphi(k)s_{u+3-k}
\Biggr), \quad u\in\mathbb{N}.
\end{eqnarray}
\item If $\{a_0=a_1=0, b_0\neq0, c_0\neq0\}$, then $\varphi
(0)=0$, $\varphi(1)=(3-\mathbb{E}S)/(1/a_2+c_0)$, and the recursion
formula \eqref{4A} is satisfied.
\item If $\{a_0\neq0, b_0=b_1=0, c_0\neq0\}$, then $\varphi
(0)=0$, $\varphi(1)=(3-\mathbb{E}S)/c_0$, and the same recursion
formula \eqref{4A} holds.
\item If $\{a_0\neq0, b_0\neq0, c_0=c_1=0\}$, then $\varphi
(0)=3-\mathbb{E}S$, $\varphi(1)=(3-\mathbb{E}S)/s_2$, and the recursion
formula \eqref{4C} is satisfied.
\end{itemize}
\end{thm}

We observe that all formulas presented in Theorem~\ref{t5} can be used
to cal\-culate numerical values of survival or ruin probabilities for
an arbitrary three-seasonal risk model and for an arbitrary initial
surplus value $u$. The algorithms based on the derived relations work
quite quickly and accurately. A few numerical examples for calculating
ruin probability in the various versions of the three-seasonal risk
model are presented in Section~\ref{papa}.

\section{Net profit condition}\label{pa}

In this section, we present a proof of Theorem~\ref{t4}. We recall that
we denote the ultimate survival probability by $\varphi(u)$.

%\textbf{Proof of Theorem~\ref{t4}}.
\begin{proof}[Proof of Theorem~\ref{t4}]
For an arbitrary $u\in\mathbb
{N}_0$, we have
\begin{align}
\label{main eq} \varphi(u)&=\mathbb{P} \Biggl(\bigcap\limits_{n=1}^{\infty}\Biggl\{u+n-\sum_{i=1}^{n}Z_i>0\Biggr\} \Biggr)\nonumber\\
&= \mathbb{P} \Biggl(\bigcap\limits_{n=3}^{\infty} \Biggl\{u+n-\sum_{i=1}^{n}Z_i>0 \Biggr\} \cap\{Z_1 \geqslant u+1 \}\cap\{Z_1+Z_2\geqslant u+2 \} \Biggr)\nonumber\\
&\quad - \mathbb{P} \Biggl(\bigcap\limits_{n=3}^{\infty} \Biggl\{u+n-\sum_{i=1}^{n}Z_i>0 \Biggr\} \cap\{Z_1 \geqslant u+1 \} \Biggr)\nonumber\\
&\quad - \mathbb{P} \Biggl(\bigcap\limits_{n=3}^{\infty} \Biggl\{u+n-\sum_{i=1}^{n}Z_i>0 \Biggr\} \cap\{Z_1+Z_2 \geqslant u+2\} \Biggr)\nonumber\\
&\quad +\mathbb{P} \Biggl(\bigcap\limits_{n=3}^{\infty} \Biggl\{u+n-\sum_{i=1}^{n}Z_i>0 \Biggr\}\Biggr).
\end{align}
Since the model is three-seasonal, the last probability in (\ref{main
eq}) can be expressed by the sum
\begin{align}
\label{main1}
\sum_{k=0}^{u+2}s_k\ \mathbb{P} \Biggl(\bigcap\limits_{n=1}^{\infty} \Biggl\{u+n+3-k-\sum_{i=1}^{n}Z_i>0\Biggr\} \Biggr)=\sum_{k=0}^{u+2}s_k \ \varphi(u+3-k),
\end{align}
where, as before, $s_k=\mathbb{P}(Z_1+Z_2+Z_3=k)$ for $ k\in\mathbb
{N}_0$.\\
The second probability in (\ref{main eq}) is equals
\begin{align}
\label{main2}
&\sum_{k=u+1}^{\infty}a_k\ \mathbb{P} \Biggl(\bigcap\limits_{n \geqslant3} \Biggl\{u+n-k-Z_2-Z_3-\sum_{i=4}^{n}Z_i>0 \Biggr\}\Biggr)\nonumber\\
&\quad =a_{u+2} \ \mathbb{P}(Z_2+Z_3=0) \ \mathbb{P}\Biggl(\bigcap\limits_{n=4}^{\infty}\Biggl\{n-2-\sum_{i=4}^{n}Z_i>0 \Biggr\} \Biggr)\nonumber\\
&\qquad +a_{u+1} \ \mathbb{P}(Z_2+Z_3=0) \ \mathbb{P}\Biggl(\bigcap\limits_{n=4}^{\infty}\Biggl\{n-1-\sum_{i=4}^{n}Z_i>0 \Biggr\} \Biggr)\nonumber\\
&\qquad +a_{u+1} \ \mathbb{P}(Z_2+Z_3=1) \ \mathbb{P}\Biggl(\bigcap\limits_{n=4}^{\infty}\Biggl\{n-2-\sum_{i=4}^{n}Z_i>0 \Biggr\} \Biggr)\nonumber\\
&\quad =a_{u+2}b_0c_0\varphi(1)+a_{u+1}b_0c_0\varphi(2)+a_{u+1}b_0c_1\varphi(1)+a_{u+1}b_1c_0\varphi(1).
\end{align}
Similarly, the third probability in (\ref{main eq}) is
\begin{align}
\label{main3} &
\mathbb{P}(Z_1+Z_2=u+2) \ \mathbb{P}(Z_3=0) \ \mathbb {P} \Biggl(\bigcap \limits_{n=4}^{\infty}\Biggl\{n-2-\sum_{i=4}^{n}Z_i>0\Biggr\} \Biggr)\nonumber\\
&\quad =c_0\varphi(1)\sum_{k=0}^{u+2}a_kb_{u+2-k},
\end{align}
and, finally, the first probability in (\ref{main eq}) is
\begin{align}
\label{main4}
&\mathbb{P}(Z_1=u+1) \ \mathbb{P}(Z_2=1)\ \mathbb{P}(Z_3=0) \ \mathbb {P} \Biggl(\bigcap\limits_{n=4}^{\infty}\Biggl\{n-2-\sum_{i=4}^{n}Z_i>0\Biggr\} \Biggr)\nonumber\\
&\qquad +\mathbb{P}(Z_1=u+2) \ \mathbb{P}(Z_2=0) \ \mathbb {P}(Z_3=0) \ \mathbb{P} \Biggl(\bigcap\limits_{n=4}^{\infty}\Biggl\{n-2-\sum_{i=4}^{n}Z_i>0\Biggr\} \Biggr)\nonumber\\
&\quad =a_{u+1}b_1c_0\varphi(1)+a_{u+2}b_0c_0\varphi(1).
\end{align}
Substituting (\ref{main1})--(\ref{main4}) into (\ref{main eq}), we get that
\begin{align}
\label{main5}
\varphi(u)&=\sum_{k=0}^{u+2}s_k\ \varphi(u+3-k)-a_{u+1}b_0c_0\varphi (2)\nonumber\\
&\quad -a_{u+1}b_0c_1\varphi(1)-c_0\varphi(1)\sum_{k=0}^{u+2}a_k b_{u+2-k}
\end{align}
for all $u\in\mathbb{N}_0$.

Therefore, for $v\in\mathbb{N}_0$, we have
\begin{align}
\label{main6}
\sum_{u=0}^{v}\varphi(u)&=\sum_{u=0}^{v}\sum_{k=0}^{u+2}s_k\varphi (u+3-k)\nonumber\\
&\quad -b_0c_0\varphi(2) \bigl({A}(v+1)-a_0\bigr) -b_0c_1\varphi(1) \bigl({A}(v+1)-a_0\bigr)\nonumber\\
&\quad -c_0\varphi (1)\sum_{u=0}^{v}\sum_{k=0}^{u+2}a_k b_{u+2-k}.
\end{align}
We observe that the sum
\[
\sum_{u=0}^{v}\sum
_{k=0}^{u+2}a_k b_{u+2-k}
\]
can be rewritten in the form
\begin{align}
\label{main7}
&a_0\sum_{u=0}^{v}b_{u+2}+a_1\sum_{u=0}^{v}b_{u+1}+a_2\sum_{u=0}^{v}b_{u}+\sum_{k=3}^{v+2}a_{k}\sum_{u=k-2}^{v}b_{u+2-k}\nonumber\\
&\quad =a_0\bigl({B}(v+2)-b_0-b_1\bigr)+a_1\bigl({B}(v+1)-b_0\bigr)+a_2{B}(v)\nonumber\\
&\qquad +\sum_{k=3}^{v+2}a_{k} {B}(v+2-k)\nonumber\\
&\quad =\sum_{k=0}^{v+2}a_{k} {B}(v+2-k)-a_0b_0-a_0b_1-a_1b_0.
\end{align}
Similarly, the sum
\[
\sum_{u=0}^{v}\sum
_{k=0}^{u+2}s_k\varphi(u+3-k)
\]
equals
\begin{align}
\label{main8} &\sum_{k=1}^{v+3}\varphi(k)
D(v+3-k)-s_0\varphi(1)-s_1\varphi (1)-s_0
\varphi(2),
\end{align}
where\vspace{-2mm}
\[
{D}(x)=\sum_{k=0}^{\lfloor x\rfloor}s_k=\sum
\limits
_{k=0}^{\lfloor
x\rfloor}\mathbb{P}(Z_1+Z_2+Z_3=k).
\]
Relations (\ref{main6}), (\ref{main7}), and (\ref{main8}) imply that
\begin{align*}
\sum_{k=0}^{v}\varphi(k)&=\sum_{k=1}^{v+3}\varphi(k) D(v+3-k)\\
&\quad -b_0c_0\varphi(2)A(v+1)-b_0c_1\varphi(1)A(v+1)\\
&\quad -c_0\varphi(1)\sum_{k=0}^{v+2}a_k B(v+2-k)
\end{align*}
or, equivalently,\vspace{-2mm}
\begin{align}
\label{main9}
&\sum_{k=0}^{v+3}\varphi(k)\bigl(1-D(v+3-k)\bigr)\nonumber\\
&\quad =\varphi(v+1)+\varphi (v+2)+\varphi(v+3)-\varphi(0)D(v+3)-b_0c_0\varphi(2)A(v+1)\nonumber\\
&\qquad -b_0c_1\varphi (1)A(v+1) +c_0\varphi(1) \Biggl( \sum_{k=0}^{v+2}a_k \overline {B}(v+2-k)+A(v+2)\Biggr).
\end{align}
For each $K\in[1,v+2),$ we have
\begin{align*}
\sum_{k=0}^{v+2}a_k
\overline{B}(v+2-k)=\sum_{k=0}^{K}a_k
\overline {B}(v+2-k)+\sum_{k=K+1}^{v+2}a_k
\overline{B}(v+2-k).
\end{align*}
Therefore,\vspace{-4mm}
\begin{align*}
\limsup\limits
_{v\rightarrow\infty}\sum_{k=0}^{v+2}a_k
\overline {B}(v+2-k) \leqslant\sum_{k=K+1}^{\infty}a_k
\end{align*}
for each $K \geqslant1$, and so\vspace{-2mm}
\begin{align}
\label{main10} \lim\limits
_{v\rightarrow\infty}\sum_{k=0}^{v+2}a_k
\overline{B}(v+2-k)=0.
\end{align}

The sequence $\varphi(u), u \in\mathbb{N}_0$, is bounded and
nondecreasing. So, the limit
$\varphi(\infty):=
\lim\limits_{u\rightarrow\infty}\varphi(u).
$
Similarly to the derivation of (\ref{main10}), we can get that
\begin{align*}
\limsup\limits
_{v\rightarrow\infty}\sum_{k=0}^{v+3}\bigl(
\varphi(\infty )-\varphi(k)\bigr) \bigl(1-D(v+3-k)\bigr) \leqslant\sup
\limits_{k \geqslant
K+1}
\bigl(\varphi(\infty)-\varphi(k)\bigr) \mathbb{E}S
\end{align*}
for each fixed $K \geqslant1$. Therefore,
\begin{align}
\label{main11}
\lim\limits_{v\rightarrow\infty}\sum_{k=0}^{v+3}\varphi (k) \bigl(1-D(v+3-k)\bigr)&=\varphi(\infty)\lim\limits_{v\rightarrow\infty}\sum_{k=0}^{v+3}\bigl(1-D(k)\bigr)\nonumber\\
&=\varphi(\infty)\mathbb{E}S.
\end{align}
Relations (\ref{main9}), (\ref{main10}), and (\ref{main11}) imply that
\begin{align}
\label{main12} \varphi(\infty) (3-\mathbb{E}S)=\varphi(0)+b_0c_0
\varphi(2)+b_0c_1\varphi (1)+c_0\varphi(1).
\end{align}
Now we consider the last equality and examine all possible cases.
\begin{description}
\item[(I)] If $\mathbb{E}S>3$, then (\ref{main12}) implies that
$\varphi(\infty)=0$ because the left side of (\ref{main12}) is
nonnegative in all cases. So, in this case, $\psi(u)=1$ for all $u \in
\{0, 1, 2, \ldots\}$.

\item[(II)] If $\mathbb{E}S=3$, then (\ref{main12}) implies that
\begin{align*}
\varphi(0)+b_0c_0\varphi(2)+b_0c_1
\varphi(1)+c_0\varphi(1)=0.
\end{align*}
Additionally, in this situation we have that $s_3=1$ or $s_3<1$.

\item[(II-A)] If $s_3=1$, then we have
\begin{equation*}
\begin{cases}
a_0b_0c_0=0,\\
a_1b_0c_0+a_0b_1c_0+a_0b_0c_1=0, \\
a_0b_0c_2+a_0b_2c_0+a_2b_0c_0+a_1b_1c_0+a_0b_1c_1+a_1b_0c_1=0, \\
a_0b_0c_3+a_0b_3c_0+a_3b_0c_0+a_1b_2c_0+a_2b_1c_0\\
\hspace{30pt}+a_0b_1c_2+a_0b_2c_1+a_1b_0c_2+a_2b_0c_1+a_1b_1c_1=1,\\
\varphi(0)+b_0c_0\varphi(2)+b_0c_1\varphi(1)+c_0\varphi(1)=0.
\end{cases} %
\end{equation*}
\end{description}
Taking into account that all numbers $a_k$, $b_k$, and $c_k$ are local
probabilities for all $k\in\mathbb{N}_0$, the last system implies the
following possible cases.
\begin{description}
\item[(a)] $\{a_3=b_0=c_0=1\}$ and $\varphi(0)=\varphi
(1)=\varphi(2)=0$.
In this case, $\psi(0)=\psi(1)=\psi(2)=1$ and $\psi(u)=0, u \in\{3, 4,
\ldots\}$ because
\begin{equation*}
W_u(n)= %
\begin{cases}
u-2  &\text{ if} \ \  n\equiv1 \ \text{mod} \ 3,\\
u-1  &\text{ if} \ \  n\equiv2 \ \text{mod} \ 3, \\
u    &\text{ if} \ \  n\equiv0 \ \text{mod} \ 3.
\end{cases} %
\end{equation*}
\item[(b)] $\{a_0=b_3=c_0=1\}$ and $\varphi(0)=\varphi(1)=0$.
In this case, $\psi(0)=\psi(1)=1$ and $\psi(u)=0$ for $u \in\{2, 3,
\ldots\}$ because
\begin{equation*}
W_u(n)= %
\begin{cases}
u+1 &\text{ if} \ \  n\equiv1 \ \text{mod} \ 3,\\
u-1 &\text{ if} \ \  n\equiv2 \ \text{mod} \ 3, \\
u   &\text{ if} \ \  n\equiv0 \ \text{mod} \ 3.
\end{cases} %
\end{equation*}
\item[(c)] $\{a_0=b_0=c_3=1\}$ and $\varphi(0)=0$.
In this case, $\psi(0)=1$ and $\psi(u)=0$ for $u \in\{1,2, \ldots\}$
because
\begin{equation*}
W_u(n)= %
\begin{cases}
u+1  &\text{ if} \ \  n\equiv1 \ \text{mod} \ 3,\\
u+2  &\text{ if} \ \  n\equiv2 \ \text{mod} \ 3, \\
u    &\text{ if} \ \  n\equiv0 \ \text{mod} \ 3.
\end{cases} %
\end{equation*}
\item[(d)] $\{a_2=b_1=c_0=1\}$ and $\varphi(0)=\varphi(1)=0$.
In this case, $\psi(0)=\psi(1)=1$ and $\psi(u)=0$ for $u \in\{2,3,
\ldots\}$.
\item[(e)] $\{a_1=b_2=c_0=1\}$ and $\varphi(0)=\varphi(1)=0$.
In this case, $\psi(0)=\psi(1)=1$ and $\psi(u)=0$ for $u \in\{2,3,
\ldots\}$.
\item[(f)] $\{a_0=b_2=c_1=1\}$ and $\varphi(0)=0$.
In this case, $\psi(0)=1$ and $\psi(u)=0$ for $u \in\{1,2, \ldots\}$.
\item[(g)] $\{a_0=b_1=c_2=1\}$ and $\varphi(0)=0$.
In this case, $\psi(0)=1$ and $\psi(u)=0$ for $u \in\{1,2, \ldots\}$.
\item[(h)] $\{a_2=b_0=c_1=1\}$ and $\varphi(0)=\varphi(1)=0$.
In this case, $\psi(0)=\psi(1)=1$ and $\psi(u)=0$ for $u \in\{2,3,
\ldots\}$.
\item[(i)] $\{a_1=b_0=c_2=1\}$ and $\varphi(0)=0$.
In this case, $\psi(0)=1$ and $\psi(u)=0$ for $u \in\{1,2, \ldots\}$.
\item[(j)] $\{a_1=b_1=c_1=1\}$ and $\varphi(0)=0$.
In this case, $\psi(0)=1$ and $\psi(u)=0$ for $u \in\{1,2, \ldots\}$.
\item[(II-B)] If $s_3<1$ and $\mathbb{E}S=3$, then it is necessary
that $s_0 \neq0$ or $s_1 \neq0$ or $s_2 \neq0$ because, on the
contrary, $\mathbb{E}S=3s_3+4s_4+5s_5+ \dots>3(s_3+s_4+ \cdots)=3$.
In this situation, it suffices to consider the following cases:
\[
\{s_0 \neq0\}, \qquad \{s_0=0, s_1 \neq0\}, \qquad \{s_0=0, s_1=0, s_2 \neq 0\}.
\]
\item[(k)] If $s_0=a_0b_0c_0 \neq0$, then (\ref{main12})
implies that $\varphi(0)=\varphi(1)=\varphi(2)=0$, and from (\ref
{main5}) we obtain $\varphi(u)=0$ for all $u \in\{3,4, \ldots\}$. So,
$\psi(u)=1$ if $u \in\{0,1, \ldots\}$ in this case.
\item[(l)] If $s_0=a_0b_0c_0 = 0$ and
$s_1=a_0b_0c_1+a_0b_1c_0+a_1b_0c_0 \neq0,$ then we have the following
possible cases.
\item[(l-1)] $\{a_0=0$, $a_1 \neq0$, $b_0 \neq0$, $c_0 \neq
0\}$.
In this situation, Eq.~(\ref{main12}) implies that $\varphi(0)=\varphi
(1)=\varphi(2)=0$, whereas (\ref{main5}) implies that $\varphi(u)=0$
for all $u \in\{3,4, \ldots\}$. So, $\psi(u)=1$ for all $u \in\{0,1,
\ldots\}$ in this case.
\item[(l-2)] $\{a_0 \neq0$, $b_0=0$, $b_1 \neq0$, $c_0 \neq
0\}$.
In this situation, Eq.~(\ref{main12}) implies that $\varphi(0)=\varphi
(1)=0$, and (\ref{main5}) implies that $\varphi(u)=0$ for all $u \in\{
2,3, \ldots\}$. Therefore, $\psi(u)=1$ for all $u \in\{0,1, \ldots\}
$ again.
\item[(l-3)] $\{a_0 \neq0$, $b_0 \neq0$, $c_0=0$, $c_1 \neq
0\}$.

Equality (\ref{main12}) implies that $\varphi(0)=\varphi(1)=0$, and
(\ref{main5}) implies that $\varphi(u)=0$ for all $u \in\{2,3, \ldots
\}$. So, in this case, $\psi(u)=1$ for all $u \in\{0,1, \ldots\}$.
\item[(m)] If $s_0=a_0b_0c_0 = 0$,
$s_1=a_0b_0c_1+a_0b_1c_0+a_1b_0c_0=0$ and
$s_2=a_0b_0c_2+a_0b_2c_0+a_2b_0c_0+a_1b_1c_0+a_1b_0c_1+a_0b_1c_1 \neq
0$, then there exist the following possible cases.
\item[(m-1)] $\{a_0=0$, $a_1=0$, $a_2 \neq0$, $b_0 \neq0$,
$c_0 \neq0\}$;
\item[(m-2)] $\{a_0 \neq0$, $b_0=0$, $b_1=0$, $b_2 \neq0$,
$c_0 \neq0\}$;
\item[(m-3)] $\{a_0 \neq0$, $b_0 \neq0$, $c_0=0$, $c_1=0$,
$c_2 \neq0\}$;
\item[(m-4)] $\{a_0=0$, $a_1 \neq0$, $b_0=0$, $b_1 \neq0$,
$c_0 \neq0\}$;
\item[(m-5)] $\{a_0=0$, $a_1\neq0$, $b_0 \neq0$, $c_0=0$,
$c_1 \neq0\}$;
\item[(m-6)] $\{a_0 \neq0$, $b_0=0$, $b_1 \neq0$, $c_0=0$,
$c_1 \neq0\}$.
\end{description}
In all cases, Eqs.~(\ref{main12}) and (\ref{main5}) imply that $\varphi
(u)=0,$ and so $\psi(u)=1$ for all $u \in\mathbb{N}_0$. Theorem~\ref
{t4} is proved.
\end{proof}

\section{Recursive formulas}\label{pa1}

In this section, we prove Theorem~\ref{t5}. Equality \eqref{main12}
from the previous section plays a crucial role.

%\textbf{Proof of Theorem~\ref{t5}}.
\begin{proof}[Proof of Theorem~\ref{t5}]
Let we consider the case $\mathbb{E}S<3$.
First, we prove that\break $\varphi(\infty)=1$. According to the definition
\begin{equation*}
\varphi(\infty)=\lim\limits
_{u\rightarrow\infty}\mathbb{P} \Biggl(\bigcap
\limits_{n=1}^{\infty}
\Biggl\{\sum_{i=1}^{n}(Z_i-1)<u
\Biggr\} \Biggr)= \lim\limits
_{u\rightarrow\infty}\mathbb{P} \Bigl(\sup\limits
_{u\geqslant
1}
\eta_n<u \Bigr),
\end{equation*}
where\vspace{-2mm}
\[
\eta_n=\sum_{i=1}^{n}(Z_i-1), \quad n\in\mathbb{N}.
\]
If $n=3N$, $N \in\mathbb{N}$, then
\begin{align*}
\frac{\eta_n}{n}=\frac{\eta_{3N}}{3N} &=\frac{1}{3} \Biggl(\frac{1}{N}\sum_{i=1}^N(Z_{3i-2}-1)
 +\frac{1}{N}\sum_{i=1}^N(Z_{3i-1}-1)+\frac{1}{N}\sum_{i=1}^N(Z_{3i}-1)\Biggr).
\end{align*}
If $n=3N+1$, $N \in\mathbb{N}$, then
\begin{align*}
\frac{\eta_n}{n}=\frac{\eta_{3N+1}}{3N+1} &=\frac{N+1}{3N+1} \ \frac{1}{N+1}\sum_{i=1}^{N+1}(Z_{3i-2}-1)\\
&\quad +\frac{N}{3N+1} \Biggl(\frac{1}{N}\sum_{i=1}^N(Z_{3i-1}-1)+\frac{1}{N}\sum_{i=1}^N(Z_{3i}-1) \Biggr).
\end{align*}
If $n=3N+2$, $N \in\mathbb{N}$, then
\begin{align*}
\frac{\eta_n}{n}=\frac{\eta_{3N+2}}{3N+2}&=\frac{N}{3N+2} \ \frac{1}{N}\sum_{i=1}^N(Z_{3i}-1)\\
&\quad +\frac{N+1}{3N+2} \Biggl(\frac{1}{N+1}\sum_{i=1}^{N+1}(Z_{3i-2}-1)+\frac{1}{N+1}\sum_{i=1}^{N+1}(Z_{3i-1}-1) \Biggr).
\end{align*}
Hence, the strong law of large numbers implies that
\begin{align*}
\frac{\eta_n}{n}\mathop{\rightarrow}\limits
_{n\rightarrow\infty}\frac{1}{3}(
\mathbb{E}Z_1-1+\mathbb{E}Z_2-1+\mathbb{E}Z_3-1)=
\frac{\mathbb{E}S-3}{3}
\end{align*}
almost surely.

From this it follows that
\begin{align}
\label{main14} \mathbb{P} \biggl(\sup_{m\geqslant n} \Big|
\frac{\eta_m}{m}+\mu \Big|<\frac{\mu}{2} \biggr) \mathop{\rightarrow}\limits
_{n\rightarrow\infty
}\ 1
\end{align}
with $\mu:=(\mathbb{E}S-3)/3 >0$.

For arbitrary positive $u$ and $N \in\mathbb{N}$, we have
\begin{align*}
\mathbb{P} \Bigl(\sup_{n\geqslant1} \eta_n <u \Bigr) &
\geqslant\mathbb {P} \Biggl( \Biggl(\bigcap\limits
_{n=1}^{N}
\biggl\{\eta_n\leqslant\frac
{u}{2} \biggr\} \Biggr)\cap
\Biggl(\bigcap\limits
_{n=N+1}^{\infty} \biggl\{
\eta_n\leqslant\frac
{u}{2} \biggr\} \Biggr) \Biggr)
\\
&\geqslant\mathbb{P} \Biggl(\bigcap\limits
_{n=1}^{N} \biggl
\{\eta_n\leqslant \frac{u}{2} \biggr\} \Biggr)+\mathbb{P}
\Biggl(\bigcap\limits
_{n=N+1}^{\infty} \{\eta _n
\leqslant0 \} \Biggr)-1
\\
&\geqslant\mathbb{P} \Biggl(\bigcap\limits
_{n=1}^{N} \biggl
\{\eta_n\leqslant \frac{u}{2} \biggr\} \Biggr)+\mathbb{P}
\biggl(\sup_{m \geqslant N+1} \Big|\frac{\eta_m}{m}+ \mu \Big| <
\frac{\mu}{2} \biggr)-1.
\end{align*}
The last inequality implies that
\begin{align*}
\lim\limits
_{u\rightarrow\infty} \mathbb{P} \Bigl(\sup_{n\geqslant1}
\eta_n <u \Bigr) \geqslant \mathbb{P} \biggl(\sup
_{m\geqslant N+1} \Big|\frac{\eta_m}{m}+\mu \Big| <\frac{\mu}{2} \biggr)
\end{align*}
for arbitrary $N \in\mathbb{N}$.

Hence, according to (\ref{main14}), we have that $\varphi(\infty)=1$.

Substituting this into (\ref{main12}), we get
\begin{align}
\label{main15} 3-\mathbb{E}S=\varphi(0)+b_0c_0
\varphi(2)+b_0c_1\varphi(1)+c_0\varphi(1).
\end{align}
In addition, Eq.~(\ref{main5}) can be rewritten as follows:
\begin{align}
\label{main16}
\varphi(u)&=\sum_{k=0}^{u+2}s_{u+2-k}\ \varphi (k+1)-a_{u+1}b_0c_0\varphi(2)-a_{u+1}b_0c_1\varphi(1)\nonumber\\
&\quad -c_0\varphi(1)\sum_{k=0}^{u+2}a_k b_{u+2-k}, \quad u \in\mathbb{N}_0.
\end{align}
Now we consider the last two formulas to get a suitable recursion
procedure described in Theorem~\ref{t5}.
\begin{itemize}
\item First, let $s_0=a_0b_0c_0 \neq0$, and let the
sequences $\alpha_n$, $\beta_n$, $\gamma_n$ be defined in the statement
of Theorem~\ref{t5}.
\end{itemize}

We prove \eqref{opop} by induction.
We observe that relation (\ref{main15}) implies immediately:
\begin{align*}
\varphi(0)&=\alpha_0\varphi(0)+\beta_0\varphi(1)+\gamma_0(3-\mathbb {E}S),\\
\varphi(1)&=\alpha_1\varphi(0)+\beta_1\varphi(1)+\gamma_1(3-\mathbb {E}S),\\
\varphi(2)&=\alpha_2\varphi(0)+\beta_2\varphi(1)+\gamma_2(3-\mathbb{E}S).
\end{align*}
Now suppose that Eq.~(\ref{opop}) holds for all $n =0,1,\ldots,N-1$,
and we will prove that (\ref{opop}) holds for $n =N$.
By (\ref{main16}) we have
\begin{align*}
\varphi(N-3)&=\sum_{k=0}^{N-1}s_{N-1-k}\ \varphi (k+1)-a_{N-2}b_0c_0\varphi(2)-a_{N-2}b_0c_1\varphi(1)\\
&\quad -c_0\varphi(1)\sum_{k=0}^{N-1}a_kb_{N-1-k}.
\end{align*}
Therefore, using the induction hypothesis, we get
\begin{align}
\label{main18}
s_0\varphi(N)&=\varphi(N-3)-\sum_{k=1}^{N-1}s_{k}\varphi(N-k)+a_{N-2}b_0c_0\varphi(2)\nonumber\\
&\quad +a_{N-2}b_0c_1\varphi(1)+c_0\varphi(1)\sum_{k=0}^{N-1}a_kb_{N-1-k}\nonumber\\
&=\alpha_{N-3}\varphi(0)+\beta_{N-3}\varphi(1)+\gamma_{N-3}(3-\mathbb {E}S)\nonumber\\
&\quad -\sum_{k=1}^{N-1}s_{k}\bigl(\alpha_{N-k}\varphi(0) +\beta_{N-k}\varphi(1)\nonumber\\
&\quad +\gamma_{N-k}(3-\mathbb{E}S)\bigr)+a_{N-2}b_0c_0\varphi (2)+a_{N-2}b_0c_1\varphi(1)\nonumber\\
&\quad +c_0\varphi(1)\sum_{k=0}^{N-1}a_kb_{N-1-k}.
\end{align}
Since
\begin{align*}
\varphi(2)=-\frac{1}{b_0c_0}\varphi(0)-\frac{c_1}{c_0}\varphi(1)-\frac{1}{b_0}\varphi(1)+\frac{3-\mathbb{E}S}{b_0c_0}
\end{align*}
due to (\ref{main15}), we obtain from (\ref{main18}) that
\begin{align*}
\varphi(N)&=\varphi(0) \ \frac{1}{s_0} \Biggl(\alpha_{N-3}-\sum_{k=1}^{N-1}s_k\alpha_{N-k}-a_{N-2} \Biggr)\\
&\quad +\varphi(1) \ \frac{1}{s_0} \Biggl(\beta_{N-3}-\sum_{k=1}^{N-1}s_k\beta_{N-k}-a_{N-2}c_0+c_0\sum_{k=0}^{N-1}a_kb_{N-1-k}\Biggr)\\
&\quad +(3-\mathbb{E}S) \ \frac{1}{s_0} \Biggl(\gamma_{N-3}-\sum_{k=1}^{N-1}s_k\gamma_{N-k}+a_{N-2} \Biggr)\\
&=\alpha_N\varphi(0)+\beta_N\varphi(1)+\gamma_N(3-\mathbb{E}S).
\end{align*}
Hence, the desired relation (\ref{opop}) holds for all $n \in\mathbb
{N}_0$ by induction.

\begin{itemize}
\item If $\{a_0=0,b_0\neq0, c_0\neq0, a_1\neq0\},$ then
$s_0=0 $ and $s_1\neq0$. Equality \eqref{main16} with $u=0$ implies
that $\varphi(0)=0$. The recursive relation \eqref{4+} can be derived
from the basic equalities \eqref{main15} and \eqref{main16} in the same
manner as relation \eqref{opop}.

\item If $\{a_0\neq0,b_0= 0, c_0\neq0, b_1\neq0\},$ then
it follows from Eq.~\eqref{main15} that
$3-\mathbb{E}S=\varphi(0)+c_0\varphi(1).$
Hence, $\varphi(1)=\tilde{\alpha}_1\varphi(0)+\tilde{\gamma
}_1(3-\mathbb{E}S)$. This is Eq.~\eqref{4++} for $n=1$. The validity of
\eqref{4++} for the other $n$ can be derived from \eqref{main16} using
the induction arguments.

\item In the case $\{a_0\neq0,b_0\neq0, c_0=0, c_1\neq0\}
$, formula \eqref{4+++} follows from \eqref{main15} if $n=1$. For the
other $n$, formula \eqref{4+++} follows from \eqref{main16} again by
using the induction arguments.

\item In the case $\{a_0=0, b_0=0, c_0\neq0\}$, we have
that $s_0=s_1=0$ and $s_2\neq0$ because of $\mathbb{E}S<3$. It follows
immediately from \eqref{main16} that $\varphi(0)=0$, whereas from \eqref
{main15} it follows that $\varphi(1)=(3-\mathbb{E}S)/s_2$. Finally, we
can get the recursive formula \eqref{4A} from \eqref{main16} using the
same induction procedure.

\item In the case $\{a_0=0, b_0\neq0, c_0= 0\}$, similarly
as in the previous one, we derive that $\varphi(0)=s_2\varphi(1)$ from
\eqref{main16}, we derive that $3-\mathbb{E}S=\varphi(0)+b_0c_1\varphi
(1)$ from \eqref{main15}, and we derive the desired formula \eqref{4B}
again from \eqref{main16}.

\item The case $\{a_0\neq0, b_0=0, c_0= 0\}$ is considered
completely analogously as both previous cases. Here we omit details.
\end{itemize}

We have that $\mathbb{E}S<3$. So, it remains to study the following
possible cases:
\begin{align*}
&\{a_0=a_1=0, b_0\neq0, c_0 \neq0\}, \qquad \{a_0\neq0, b_0=b_1=0, c_0\neq0\},\\
&\{a_0\neq0, b_0\neq0, c_0=c_1=0\}.
\end{align*}
In all these cases, the presented recursion relations follow from
Eq.~\eqref{main16}, and the initial values of survival probability
$\varphi(0)$ and $\varphi(1)$ can be obtained using Eq.~\eqref{main15}
together with Eq.~\eqref{main16} with $u=0$ or $u=1$. Theorem~\ref{t5}
is proved.
\end{proof}

\section{Numerical examples}\label{papa}

In this section, we present three examples of computing numerical
values of the finite-time ruin probability and the ultimate ruin
probability. All calculations are carried out using software
MATHEMATICA. In all presented tables, the numbers are rounded up to
three decimal places.

\smallskip

\noindent\textit{First example.} Suppose that the
three-seasonal discrete-time risk model is generated by the
r.v.s\vspace{-6pt}
\begin{table}[h]
\begin{eqnarray*}
\begin{tabular}{c|c|c|c}
$Z_1$&0&1&2\\
\hline
$\mathbb{P}$&$0.5$&$0.25$&$0.25$
\end{tabular} %
\,,\qquad\quad\!   %
\begin{tabular} {c|c|c|c} $Z_2$&0&1&2
\\
\hline$\mathbb{P}$&$0.4$&$0.3$&$0.3$ \end{tabular} %
\,,\qquad\quad\!   %
\begin{tabular} {c|c|c|c} $Z_3$&0&1&2
\\
\hline$\mathbb{P}$&$0.3$&$0.35$&$0.35$ \end{tabular} %
\,.
\end{eqnarray*}\vspace{-6pt}
\end{table}

In Table 1, we give the finite-time ruin probabilities for initial
surpluses $u\in\{0,1, \ldots10,20\}$ and times $T\in\{1,2,\ldots,
10,20\} $ together with the ultimate ruin probabilities for the same $u$.

\begin{table}[b]
\tabcolsep=4.1pt
\caption{Ruin probabilities for the first model}
%\begin{tabular} {|c|c|c|c|c|c|c|c|c|c|c|c|c|c|c|} \hline
\begin{tabular*}{\textwidth}{@{\extracolsep{\fill}}ccccccccccccc@{}}\hline
$T$&\multicolumn{12}{l}{$u$}\\
\cline{2-13}
&$0$&$1$&$2$&$3$&$4$&$5$&$6$&$7$&$8$&$9$&$10$&$20$\\
\hline
$1$&$0.5$&$0.25$&$0$&$0$&$0$&$0$&$0$&$0$&$0$&$0$&$0$&$0$\\%\hline
$2$&$0.65$&$0.325$&$0.075$&$0$&$0$&$0$&$0$&$0$&$0$&$0$&$0$&$0$\\%\hline
$3$&$0.703$&$0.404$&$0.128$&$0.026$&$0$&$0$&$0$&$0$&$0$&$0$&$0$&$0$\\%\hline
$4$&$0.733$&$0.445$&$0.169$&$0.046$&$0.007$&$0$&$0$&$0$&$0$&$0$&$0$&$0$\\%\hline
$5$&$0.751$&$0.475$&$0.2$&$0.066$&$0.014$&$0.002$&$0$&$0$&$0$&$0$&$0$&$0$\\%\hline
$6$&$0.768$&$0.503$&$0.233$&$0.089$&$0.026$&$0.005$&$0.001$&$0$&$0$&$0$&$0$&$0$\\%\hline
$7$&$0.779$&$0.523$&$0.256$&$0.106$&$0.035$&$0.009$&$0.002$&$0$&$0$&$0$&$0$&$0$\\%\hline
$8$&$0.788$&$0.538$&$0.275$&$0.122$&$0.045$&$0.014$&$0.003$&$0.001$&$0$&$0$&$0$&$0$\\%\hline
$9$&$0.796$&$0.554$&$0.295$&$0.139$&$0.056$&$0.019$&$0.006$&$0.001$&$0$&$0$&$0$&$0$\\%\hline
$10$&$0.802$&$0.566$&$0.310$&$0.152$&$0.065$&$0.024$&$0.008$&$0.002$&$0$&$0$&$0$&$0$\\%\hline
$20$&$0.836$&$0.632$&$0.402$&$0.243$&$0.138$&$0.075$&$0.038$&$0.018$&$0.008$&$0.003$&$0.001$&$0$\\%\hline
$\infty $&$0.877$&$0.722$&$0.541$&$0.404$&$0.301$&$0.224$&$0.167$&$0.125$&$0.093$&$0.069$&$0.052$&$0.003$\\
\hline
\end{tabular*} %
\end{table}
\begin{table}[t]
\tabcolsep=4.1pt
\caption{Ruin probabilities for the Poison model}
%\begin{tabular} {|c|c|c|c|c|c|c|c|c|c|c|c|c|c|c|}
\begin{tabular*}{\textwidth}{@{\extracolsep{\fill}}ccccccccccccc@{}}\hline
$T$&\multicolumn{12}{l}{$u$}\\
\cline{2-13}
&$0$&$1$&$2$&$3$&$4$&$5$&$6$&$7$&$8$&$9$&$10$&$20$\\
\hline $1$&$0.393$&$0.09$&$0.014$&$0.002$&$0$&$0$&$0$&$0$&$0$&$0$&$0$&$0$\\
$2$&$0.481$&$0.152$&$0.037$&$0.008$&$0.001$&$0$&$0$&$0$&$0$&$0$&$0$&$0$\\
$3$&$0.535$&$0.205$&$0.066$&$0.019$&$0.005$&$0.001$&$0$&$0$&$0$&$0$&$0$&$0$\\
$4$&$0.549$&$0.221$&$0.075$&$0.023$&$0.006$&$0.002$&$0$&$0$&$0$&$0$&$0$&$0$\\
$5$&$0.562$&$0.236$&$0.086$&$0.028$&$0.009$&$0.002$&$0.001$&$0$&$0$&$0$&$0$&$0$\\
$6$&$0.576$&$0.254$&$0.099$&$0.036$&$0.012$&$0.004$&$0.001$&$0$&$0$&$0$&$0$&$0$\\
$7$&$0.581$&$0.26$&$0.103$&$0.038$&$0.013$&$0.004$&$0.001$&$0$&$0$&$0$&$0$&$0$\\
$8$&$0.585$&$0.266$&$0.109$&$0.042$&$0.015$&$0.005$&$0.002$&$0.001$&$0$&$0$&$0$&$0$\\
$9$&$0.591$&$0.274$&$0.115$&$0.046$&$0.017$&$0.006$&$0.002$&$0.001$&$0$&$0$&$0$&$0$\\
$10$&$0.593$&$0.277$&$0.118$&$0.048$&$0.018$&$0.007$&$0.002$&$0.001$&$0$&$0$&$0$&$0$\\
$20$&$0.605$&$0.295$&$0.134$&$0.059$&$0.026$&$0.011$&$0.005$&$0.002$&$0.001$&$0.0003$&$0.0001$&$0$\\
$\infty $&$0.609$&$0.3$&$0.139$&$0.064$&$0.029$&$0.013$&$0.006$&$0.003$&$0.001$&$0.001$&$0.0002$&$0$\\
\hline
\end{tabular*} %
\end{table}
\begin{table}[b]
\tabcolsep=4.1pt
\caption{Ruin probabilities for the geometric model}
%\begin{tabular*}{\textwidth}{@{\extracolsep{\fill}}ccccccccccccccc@{}}\hline
\begin{tabular*}{\textwidth}{@{\extracolsep{\fill}}ccccccccccccc@{}}\hline
$T$&\multicolumn{12}{l}{$u$}\\
\cline{2-13}
&$0$&$1$&$2$&$3$&$4$&$5$&$6$&$7$&$8$&$9$&$10$&$20$\\
\hline $1$&$0.25$&$0.063$&$0.016$&$0.004$&$0.001$&$0$&$0$&$0$&$0$&$0$&$0$&$0$\\
$2$&$0.333$&$0.111$&$0.037$&$0.012$&$0.004$&$0.001$&$0$&$0$&$0$&$0$&$0$&$0$\\
$3$&$0.556$&$0.34$&$0.218$&$0.143$&$0.094$&$0.063$&$0.042$&$0.028$&$0.019$&$0.012$&$0.008$&$0$\\
$4$&$0.566$&$0.35$&$0.226$&$0.149$&$0.099$&$0.066$&$0.044$&$0.029$&$0.019$&$0.013$&$0.009$&$0$\\
$5$&$0.576$&$0.362$&$0.236$&$0.156$&$0.104$&$0.069$&$0.046$&$0.031$&$0.021$&$0.014$&$0.009$&$0$\\
$6$&$0.653$&$0.461$&$0.334$&$0.243$&$0.176$&$0.127$&$0.091$&$0.065$&$0.046$&$0.033$&$0.023$&$0.001$\\
$7$&$0.657$&$0.466$&$0.338$&$0.247$&$0.18$&$0.13$&$0.093$&$0.067$&$0.048$&$0.034$&$0.024$&$0.001$\\
$8$&$0.661$&$0.471$&$0.344$&$0.252$&$0.184$&$0.133$&$0.096$&$0.069$&$0.049$&$0.035$&$0.025$&$0.001$\\
$9$&$0.703$&$0.529$&$0.406$&$0.312$&$0.239$&$0.181$&$0.137$&$0.102$&$0.076$&$0.056$&$0.042$&$0.002$\\
$10$&$0.705$&$0.532$&$0.409$&$0.315$&$0.241$&$0.184$&$0.139$&$0.104$&$0.078$&$0.057$&$0.042$&$0.002$\\
$20$&$0.774$&$0.635$&$0.528$&$0.438$&$0.363$&$0.298$&$0.243$&$0.197$&$0.159$&$0.127$&$0.101$&$0.008$\\
$\infty $&$0.927$&$0.879$&$0.84$&$0.803$&$0.769$&$0.736$&$0.705$&$0.675$&$0.647$&$0.619$&$0.593$&$0.385$\\
\hline
\end{tabular*} %
\end{table}

Numerical values of the finite-time ruin probabilities are calculated
using the algorithm presented in Theorem~\ref{t3}, whereas the values
of the ultimate ruin probabilities are obtained using the formulas of
Theorem~\ref{t5}. Namely, first, we observe that $\mathbb{E}S=2.7$ and
$s_0\neq0$ in this case. So, Eq.~\eqref{opop} holds for an arbitrary
$n\in\mathbb{N}_0$. In\vadjust{\eject} particular,
\[
\begin{cases}
\varphi(250)=\alpha_{250}\varphi(0)+\beta_{250}\varphi(1)+0.3\gamma
_{250},\\
\varphi(251)=\alpha_{251}\varphi(0)+\beta_{251}\varphi(1)+0.3\gamma_{251},
\end{cases} %
\]
According to the first statement of Theorem~\ref{t5}, we can suppose
that $\varphi(250)=\break\varphi(251)=1$. So, we get $\varphi(0)$ and $\varphi
(1)$ from this system
after calculating the values of $\{\alpha_0,\alpha_1, \ldots\alpha
_{251}\}$, $\{\beta_0,\beta_1,\ldots,\beta_{251}\},$ and $\{\gamma
_0,\gamma_1, \ldots, \gamma_{251}\}$. Now it remains to use Eq.~\eqref
{opop} again to obtain the values $\varphi(u)=1-\psi(u)$ for initial
surplus values $u\in\{2,3,\ldots,10,20\}$.

\smallskip

\noindent\textit{Second example.} Suppose now that the
three-seasonal discrete-time risk model is generated by three Poison
distributions:
$Z_1$ with parameter $1/2$, $Z_2$ with parameter $2/3,$ and $Z_3$ with
parameter $4/5$. In Table 2, we present the finite-time ruin
probabilities for $u\in\{0,1, \ldots10,20\}$, $T\in\{1,2,\ldots,
10,20\} $ and the ultimate ruin probabilities for $u\in\{0,1, \ldots
10,20\}$. All calculations are made similarly as in the first example.

\smallskip

\noindent\textit{Third example.} We write $\xi\sim\mathcal
{G}(p)$ if \(\xi\) is a r.v.\ having the geometric distribution with
parameter $p\in(0,1)$, that is, $\mathbb{P}(\xi=k)=p(1-p)^k, k\in\mathbb
{N}_0$. Suppose that the three-seasonal risk model is generated by
r.v.s $Z_1\sim\mathcal{G}(3/4)$, $Z_2\sim\mathcal{G}(2/3),$ and $Z_3\sim
\mathcal{G}(1/3)$. In Table 3, we present the finite-time and
infinite-time ruin probabilities for this geometric model. The values
of initial surpluses $u$ and times $T$ are the same as in the previous examples.

The presented tables show that the behavior of ruin probabilities is
closely related to the structure of generating r.v.s and not only to
the global model characteristic $\mathbb{E}S$.

\section*{Acknowledgments} We would like to thank the anonymous referee
for the detailed and helpful comments on the first version of the manuscript.

\addcontentsline{toc}{section}{References}

% structpyb loaded by rokas.maliukevicius, 2015-12-22 14:07:22

\end{document}